\documentclass[11pt]{amsart}

\usepackage{amsmath,amssymb,amsthm,amsrefs,a4wide}
\usepackage{latexsym}
\usepackage{indentfirst}
\usepackage{graphicx}
\usepackage{placeins}
\usepackage{booktabs}
\usepackage{algorithm}
\usepackage{algorithmic}
\usepackage{multirow}

\theoremstyle{plain}
\newtheorem{theorem}{Theorem}

\theoremstyle{definition}

\newtheorem{example}{Example}

\theoremstyle{remark}

\newcommand{\bbm}{\begin{bmatrix}}
\newcommand{\ebm}{\end{bmatrix}}
\newcommand{\R}{\mathbb{R}}

\newcommand{\p}{\partial}

\newcommand{\pV}{p_{\mathrm{V}}}
\newcommand{\pE}{p_{\mathrm{E}}}
\newcommand{\pVE}{p_{\mathrm{VE}}}

\newcommand{\Ber}{\mathrm{Ber}}
\newcommand{\id}{\mathrm{id}}

\newcommand{\SW}{\mathrm{SW}}
\newcommand{\DF}{\mathrm{DF}}
\newcommand{\SWDF}{\mathrm{SWDF}}

\begin{document}

\title[Double Flip for Ising Models]{Double Flip Move for Ising Models with Mixed Boundary
  Conditions}

\author[]{Lexing Ying} \address[Lexing Ying]{Department of Mathematics, Stanford University,
  Stanford, CA 94305} \email{lexing@stanford.edu}

\thanks{This work is partially supported by NSF grant DMS 2011699.}



\begin{abstract}
  This note introduces the double flip move for accelerating the Swendsen-Wang algorithm for Ising
  models with mixed boundary conditions below the critical temperature. The double flip move
  consists of a geometric flip of the spin lattice followed by a spin value flip. Both the symmetric
  and approximately symmetric models are considered. We prove the detailed balance of the double
  flip move and demonstrate its empirical efficiency in mixing.
\end{abstract}

\maketitle

%----------------------------------------------------------
\section{Introduction}\label{sec:intro}

%========= problem
This note is concerned with the Monte Carlo sampling of Ising models with mixed boundary conditions.
Consider a graph $G=(V,E)$ with the vertex set $V$ and the edge set $E$. Assume that $V=I \cup B$,
where $I$ is the subset of interior vertices and $B$ the subset of boundary vertices. Throughout the
note, we use $i,j$ to denote the vertices in $I$ and $b$ for the vertices in $B$. $E$ is the set of
edges, where $ij\in E$ denotes an interior edge while $ib\in E$ denotes an edge between an interior
vertex $i$ and a boundary vertex $b$. The boundary condition is specified by $f=(f_b)_{b\in B}$. The
Ising model with the boundary condition $f$ is the following probability distribution $\pV(\cdot)$
over the configurations $s=(s_i)_{i\in I}$ of the interior vertex set $I$:
\begin{equation}
  \pV(s) \sim \exp\left( \beta \sum_{ij\in E} s_i s_j + \beta \sum_{ib\in E} s_i f_b \right).
  \label{eq:pV}
\end{equation}

One key feature of these models is that below the critical temperature the Gibbs distribution
exhibits on the macroscopic scale different profiles dictated by the boundary condition. Figure
\ref{fig:front} showcases two such examples, where black denotes $+1$ and white denotes $-1$. In
Figure \ref{fig:front}(a), the square Ising lattice has the $+1$ condition on the two vertical sides
but the $-1$ condition on the two horizontal sides. There are two dominant profiles: one contains a
large $-1$ cluster linking two horizontal sides and the other contains a large $+1$ cluster linking
two vertical sides. In Figure \ref{fig:front}(b), the Ising lattice supported on a disk has the $+1$
condition on two disjoint arcs and the $-1$ condition on the other two. The two dominant profiles
are shown in Figure \ref{fig:front}(b). Notice that in each example, due to the symmetry or
approximate symmetry of the Ising lattice and the boundary condition, the two dominant profiles have
comparable probabilities. Therefore, any effective sampling algorithm is required to visit both
profiles frequently.

\begin{figure}[h!]
  \centering
  \begin{tabular}{cc}
    \includegraphics[scale=0.55]{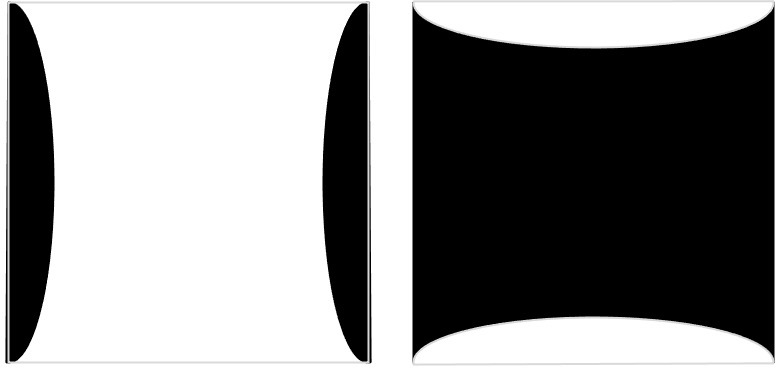} &
    \includegraphics[scale=0.55]{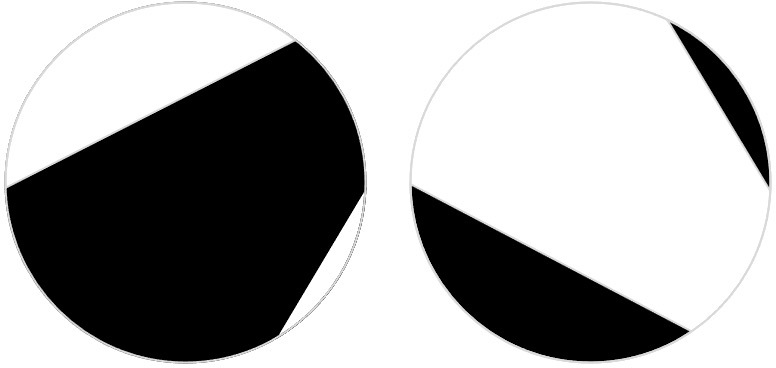} \\
    (a) & (b)
  \end{tabular}
  \caption{Ising models with mixed boundary condition. (a) a square lattice and (b) a lattice
    support on a disk. In each example, the model exhibits two dominant profiles on the macroscopic
    scale.}
  \label{fig:front}
\end{figure}

One of the most well-known methods for sampling Ising models is the Swendsen-Wang algorithm
\cite{swendsen1987nonuniversal}, which iterates between the following two substeps in each iteration
\begin{enumerate}
\item Given the current spin configuration, generate an edge configuration according the temperature
  (see Section \ref{sec:sw} for details).
\item To each connected component of the edge configuration, assign all $+1$ or all $-1$ to its
  spins with equal probability. This results in a new spin configuration.
\end{enumerate}
For many boundary conditions including the free boundary condition, the Swendsen-Wang algorithm
exhibits rapid mixing for all temperatures. However, for the mixed boundary conditions shown in
Figure \ref{fig:front}, the Swendsen-Wang algorithm experiences slow convergence under the critical
temperatures, i.e., $T < T_c$ or equivalently $\beta>\beta_c$ in terms of the inverse
temperature. The reason is that, for such a boundary condition, the energy barrier between the two
dominant profiles is much higher than the energy of these profiles. In other words, the
Swendsen-Wang algorithm needs to break a macroscopic number of edges between aligned adjacent spins
in order to transition from one dominant profile to the other. However, breaking so many edges
simultaneously is an event with exponentially small probability.

%========= contribution
In this note, we introduce the {\em double flip} move that introduces direct transitions between
these dominant profiles. When combined with the Swendsen-Wang algorithm, it accelerates the mixing
of these Ising model under the critical temperature significantly.

When the Ising model exhibits an exact symmetry (typically a reflection that negates the mixed
boundary condition), the double flip move consists of
\begin{enumerate}
\item A geometric flip of the spin lattice along a symmetry line, followed by
\item A spin-value flip at the interior vertices of the Ising model.
\end{enumerate}
The key observation is that these two flips together preserves the alignment between the adjacent
spins, hence introducing a successful Monte Carlo move.  When the Ising model exhibits only an
approximate symmetry, the double flip move consists of
\begin{enumerate}
\item A geometric flip of the spin lattice along a symmetry line, 
\item A matching step that snaps the flipped copy of the interior vertices to the original copy of
  the interior vertices, and 
\item A spin-value flip at the interior vertices of the Ising model.
\end{enumerate}
In both the exact and the approximate symmetry case, we prove the detailed balance of the double
flip move and demonstrate its efficiency when combined with the Swendsen-Wang algorithm.

%========= related work.
{\bf Related works.} In \cite{alexander2001spectral,alexander2001spectralB}, Alexander and Yoshida
studied the spectral gap of the 2D Ising models with mixed boundary conditions. More recently, in
\cite{gheissari2018effect}, Gheissari and Lubetzky studied the effect of the boundary condition for
the 2D Potts models at the critical temperature. In \cite{chatterjee2020speeding}, Chatterjee and
Diaconis showed that most of the deterministic jumps can accelerate the Markov chain mixing when the
equilibrium distribution is uniform.

%========= structure.
{\bf Contents.} The rest of the note is organized as follows. Section \ref{sec:sw} reviews the
Swendsen-Wang algorithm for the Ising models with boundary condition. Section \ref{sec:sym}
describes the double flip move for models with exact symmetry and Section \ref{sec:app} extends it
to models with approximate symmetry. Section \ref{sec:disc} discusses some future directions.

%----------------------------------------------------------
\section{Swendsen-Wang algorithm}\label{sec:sw}

In this section, we briefly review the Swendsen-Wang algorithm, which is a Markov Chain Monte Carlo
method for sampling $\pV(\cdot)$. The description here is adapted to the setting with boundary
condition. In each iteration, it generates a new configuration $t$ from the current configuration
$s$ as follows:
\begin{enumerate}
\item Generate an edge configuration $w=(w_{ij})_{ij\in E}$. If the spin values $s_i$ and $s_j$ are
  different, set $w_{ij}=0$. If $s_i$ and $s_j$ are the same, $w_{ij}$ is sampled from the Bernoulli
  distribution $\Ber(1-e^{-2\beta})$, i.e., equal to 1 with probability $1-e^{-2\beta}$ and 0 with
  probability $e^{-2\beta}$.
\item Regards all edges $ij\in E$ with $w_{ij}=1$ and $ib\in E$ with $w_{ib}=1$ as linked. Compute
  the connected components. For each connected component $\gamma$, if $\gamma$ contains a boundary
  vertex, set $(t_i)_{i\in\gamma}$ to the spin of the boundary vertex. If not, set all the spins
  $(t_i)_{i\in\gamma}$ to all $-1$ or all $+1$ with equal probability.
\end{enumerate}

The following two related probability distributions are important for analyzing the Swendsen-Wang
algorithm \cite{edwards1988generalization}. The first one is the joint vertex-edge distribution
\begin{align}
  \pVE(s,w)\sim
  & \prod_{ij\in E} \left( (1-e^{-2\beta})\delta_{s_i=s_j}\delta_{w_{ij}=1} + e^{-2\beta}\delta_{w_{ij}=0}\right) \cdot \\
  & \prod_{ib\in E} \left( (1-e^{-2\beta})\delta_{s_i=f_b}\delta_{w_{ib}=1} + e^{-2\beta}\delta_{w_{ib}=0}\right) \nonumber.
  \label{eq:pVE}
\end{align}
The second one is the edge distribution
\begin{equation}
  \pE(w) \sim
  \prod_{w_{ij}=1}(1-e^{-2\beta}) \prod_{w_{ij}=0} e^{-2\beta} \cdot 
  \prod_{w_{ib}=1}(1-e^{-2\beta}) \prod_{w_{ib}=0} e^{-2\beta} \cdot 2^{|\mathcal{C}_w|},
  \label{eq:pE}
\end{equation}
where $\mathcal{C}_w$ is set of connected components of $w$ that contain only the interior vertices.

Summing $\pVE(s,w)$ over $s$ or $w$ gives \cite{edwards1988generalization}
\begin{equation}
  \sum_s \pVE(s,w) = \pE(w),\quad
  \sum_w \pVE(s,w) = \pV(s).
  \label{eq:rel}
\end{equation}
A direct consequence of \eqref{eq:rel} is that the Swendsen-Wang algorithm can be viewed as a data
augmentation method \cite{liu2001monte}, where the first substep samples the edge configuration $w$
conditioned on the spin configuration $s$ and the second substep samples a new spin configuration
$t$ conditioned on the edge configuration $w$.

This relationship \eqref{eq:rel} also shows that Swendsen-Wang algorithm satisfies the detailed
balance, i.e.,
\[
\pV(s) P_\SW(s,t) = \pV(t) P_\SW(t,s),
\]
where $P_\SW$ is for the Swendsen-Wang transition matrix. To see this, note that $P_\SW(s,t)=\sum_w
P_w(s,t)$ and $P_\SW(t,s) = \sum_w P_w(t,s)$, where the sum in $w$ is taken over all compatible edge
configurations $w$ and $P_w(s,t)$ is the transition probability from $s$ to $t$ via the edge
configuration $w$. Therefore, it is sufficient to show
\[
\pV(s) P_w(s,t) = \pV(t) P_w(t,s)
\]
for any compatible edge configuration $w$. Since the transitions from the edge configuration $w$ to
the spin configurations $s$ and $t$ are the same, it reduces to showing
\[
\pV(s) P(s,w) = \pV(t) P(t,w),
\]
where $P(s,w)$ is the probability of obtaining the edge configuration $w$ from $s$ and similarly for
$P(t,w)$. In fact $\pV(s) P(s,w)$ is independent of the spin configuration $s$ by the following
argument. First, if an edge $ij\in E$ has configuration $w_{ij}=1$, then $s_i=s_j$. Second, if
$ij\in E$ has configuration $w_{ij}=0$, then $s_i$ and $s_j$ can either be the same or different. In
the former case, the contribution to $\pV(s) P(s,w)$ from $ij$ is $e^{2\beta} \cdot e^{-2\beta}=1$
up to a normalization constant. In the latter case, the contribution is $1\cdot 1=1$ up to the same
normalization constant. The same argument applies to the edges $ib\in E$ and therefore $\pV(s)
P(s,w)$ is independent of $s$.

The Swendsen-Wang algorithm unfortunately does not encourage transitions between the dominant
profiles shown in Figure \ref{fig:front}. For these mixed boundary condition, such a transition
requires breaking a macroscopic number of edges between aligned adjacent spins, which has an
exponentially small probability. This is the motivation for introducing the double flip move.

%----------------------------------------------------------
\section{Double flip for symmetric models}\label{sec:sym}

The double flip move is designed to introduce explicit transitions between the different profiles as
shown Figure \ref{fig:front}. This section assumes that the Ising model enjoys an explicit graph
involution, i.e., there exists a map $m: V\rightarrow V$ such that
\begin{itemize}
\item $m$ maps $I$ to $I$ and $B$ to $B$, respectively, and $m^2 = \id$,
\item $ij\in E$ iff $m(i) \sim m(j)\in E$, and $ib\in E$ iff $m(i)m(b)\in E$,
\item $f_{m(b)} = -f_b$.
\end{itemize}
For example, in Figure \ref{fig:ex1}(a) $m$ is the reflection along one of the diagonals of the
square, while in Figure \ref{fig:ex2}(a) $m$ is the reflection along either the $x$ axis or the $y$
axis.

%% \begin{figure}[h!]
%%   \centering
%%   \begin{tabular}{cc}
%%     \includegraphics[scale=0.2]{tmp.pdf} & \includegraphics[scale=0.2]{tmp.pdf}\\
%%     (a) & (b)
%%   \end{tabular}
%%   \caption{TODO}
%%   \label{fig:sym}
%% \end{figure}

In the double flip move, the first flip implements the map $m$ to the interior vertices in
$I$. After that, the second flip negates the spin of the mapped interior vertices. More
specifically, the resulting new spin configuration $t$ is defined by
\begin{equation}
  t_{m(i)} = - s_i,\quad \forall i\in I.
  \label{eq:df}
\end{equation}
Since $m^2=\id$, we also have $t_i = - s_{m(i)}$ for any $i\in I$.

\begin{theorem}
  The double flip move satisfies the detailed balance.
\end{theorem}
\begin{proof}
  To show the detailed balance, one needs to prove that, for any two spin configurations $s$ and $t$,
  \[
  \pV(s) P_\DF(s,t) = \pV(t) P_\DF(t,s),
  \]
  where $P_\DF$ is the double flip move transition matrix. From the definition of the double flip
  move, the transition probabilities $P(s,t)$ and $P(t,s)$ equal to one if $s$ and $t$ satisfy
  \eqref{eq:df} and zero otherwise. Hence it is sufficient to show $\pV(s) = \pV(t)$ when
  \eqref{eq:df} holds.

  For each $ij \in E$,
  \[
  t_i t_j = (-1)^2 s_{m(i)} s_{m(j)} = s_{m(i)} s_{m(j)}.
  \]
  Taking the sum over all interior edges gives
  \[
  \sum_{ij\in E} t_i t_j = \sum_{ij\in E} s_{m(i)} s_{m(j)}=  \sum_{ij\in E} s_i s_j.
  \]
  For each $ib\in E$,
  \[
  t_i f_b = (-1)^2 s_{m(i)} f_{m(b)} = s_{m(i)} f_{m(b)}.
  \]
  Taking the sum over all interior edges results in
  \[
  \sum_{ib\in E} t_i f_b = \sum_{ib\in E} s_{m(i)} f_{m(b)}=  \sum_{ib\in E} s_i f_b.
  \]
  Together, $\sum_{ij\in E} t_i t_j + \sum_{ib\in E} t_i f_b = \sum_{ij\in E} s_i s_j + \sum_{ib\in
    E} s_i f_b$, i.e., $\pV(s) = \pV(t)$.
\end{proof}

We can now combine the double flip move with the Swendsen-Wang move. For a constant $\eta \in(0,1)$,
each iteration of the combined algorithm performs the double flip move with probability $\eta$ and
the Swendsen-Wang move with probability $1-\eta$.
\begin{enumerate}
\item Choose $u$ from $\Ber(\eta)$, i.e., equal to 1 with probability $\eta$.
\item If $u$ is 1, perform the double flip move, else perform the Swendsen-Wang move.
\end{enumerate}

\begin{theorem}
  The Swendsen-Wang algorithm with double flip satisfies the detailed balance.
\end{theorem}
\begin{proof}
  The combined transition matrix is
  \[
  P_{\SWDF} = \eta P_{\DF} + (1-\eta) P_{\SW}.
  \]
  Section \ref{sec:sw} shows that the Swendsen-Wang algorithm satisfies the detailed balance, i.e.,
  \[
  \pV(s) P_\SW(s,t) = \pV(t) P_\SW(t,s).
  \]
  The double flip move satisfies the detailed balance as shown above,
  \[
  \pV(s) P_\DF(s,t) = \pV(t) P_\DF(t,s).
  \]
  A linear combination of these two statements give
  \[
  \pV(s)P_{\SWDF}(s,t) = \pV(t) P_{\SWDF}(t,s)
  \]
  and this finishes the proof.
\end{proof}

Below we compare the performance of the Swendsen-Wang algorithm (SW) and Swendsen-Wang with double
flip (SWDF) using two examples.

\begin{example}
  The Ising model is a square lattice. The mixed boundary condition is $+1$ at the two vertical
  sides and $-1$ at the two horizontal sides. The graph involution $m: V\rightarrow V$ is given by
  the diagonal reflection. Figure \ref{fig:ex1}(a) shows the model at size $n_1=n_2=20$.

  The experiments are performed for the problem size $n_1=n_2=100$ at the inverse temperature
  $\beta=0.5$.  We start from the all $-1$ configuration and carry out $10000$ iterations for both
  SW and SWDF. For SWDF, we set the parameter $\eta=1/100$. Figure \ref{fig:ex1}(b) and (c) plot the
  average spin value
  \[
  \frac{1}{|I|} \sum_{i\in I} s_i
  \]
  of these two algorithms, respectively. Figure \ref{fig:ex1}(b) shows that SW fails to introduce
  transitions between the $-1$ and the $+1$ dominant profiles, since the average spin is always
  below $0$. On the other hand, Figure \ref{fig:ex1}(c) demonstrates that SWDF explores both
  profiles with $100$ transitions in between.
\end{example}

\begin{figure}[h!]
  \centering
  \begin{tabular}{ccc}
    \includegraphics[scale=0.3]{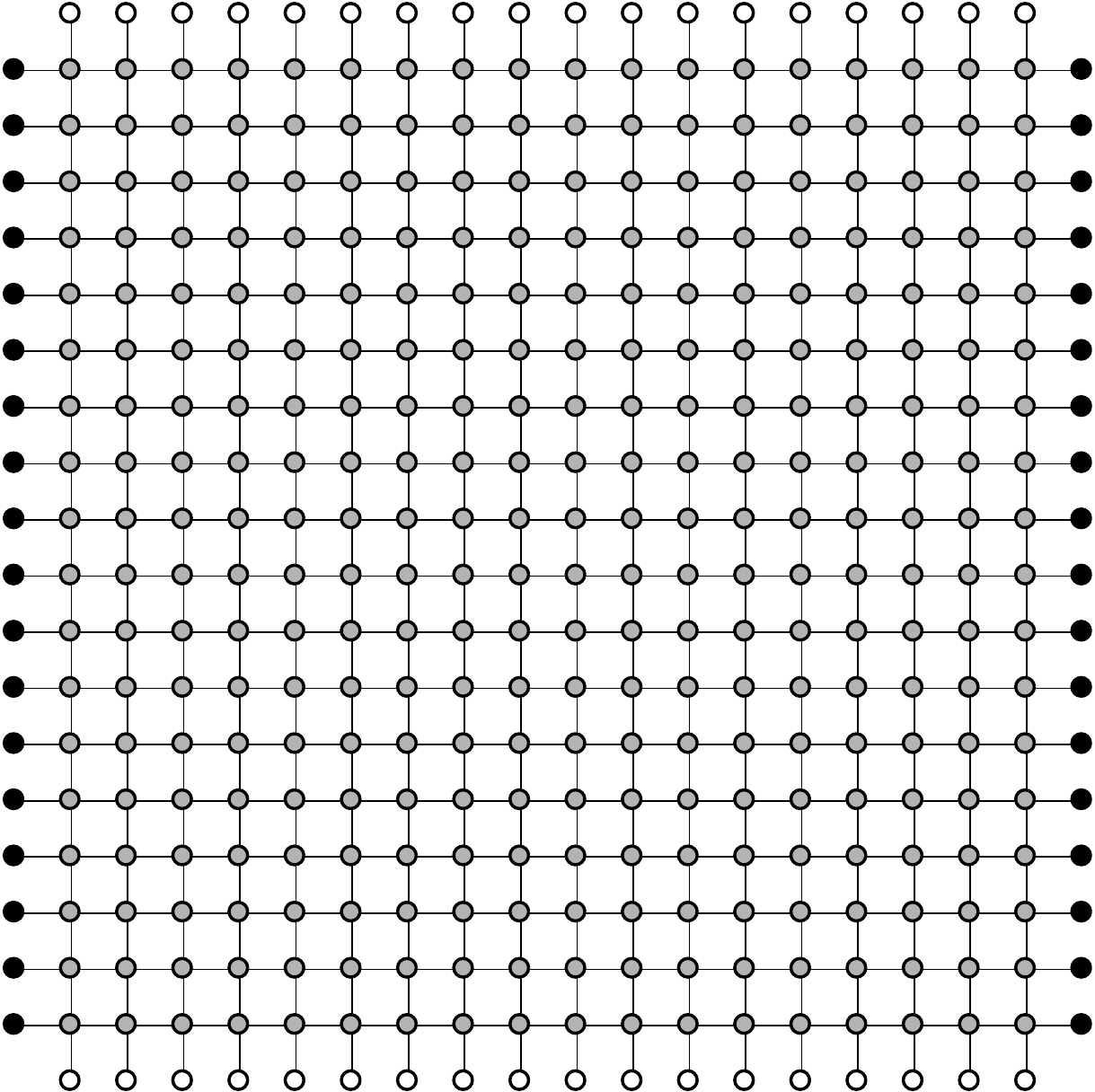} &
    \includegraphics[scale=0.3]{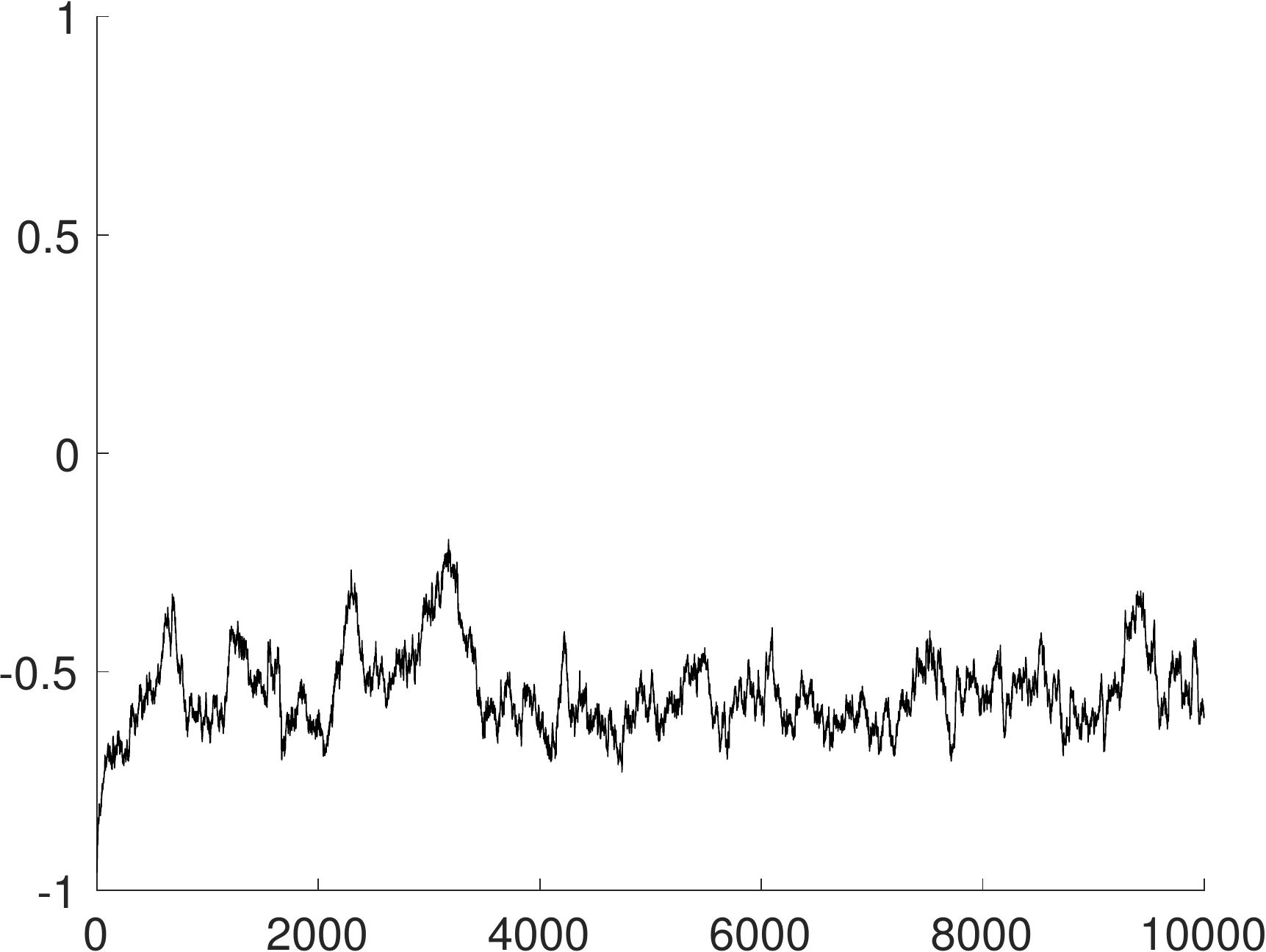} &
    \includegraphics[scale=0.3]{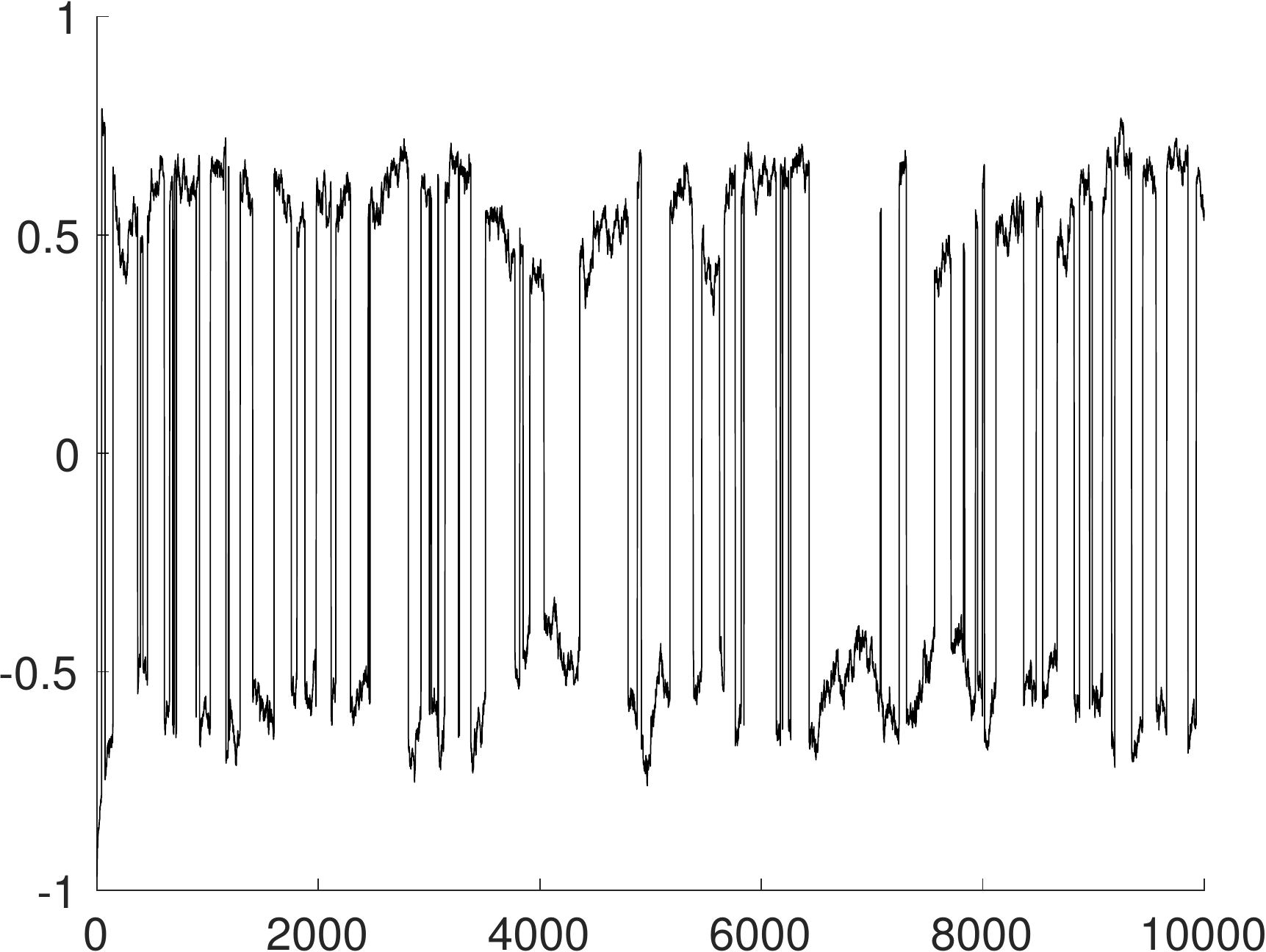} \\
    (a) & (b) & (c)
  \end{tabular}
  \caption{(a) The lattice along with the mixed boundary condition (black for $+1$ and white for
    $-1$). (b) The average spin value of the Swendsen-Wang algorithm. (c) The average spin value of
    the Swendsen-Wang algorithm with double flip.}
  \label{fig:ex1}
\end{figure}

\begin{example}
  The Ising lattice is still a square. The mixed boundary condition is $+1$ in the first and third
  quadrants but $-1$ in the second and fourth quadrants. The graph involution $m: V\rightarrow V$ is
  given by the reflection along either the $x$ or the $y$ axis.  Figure \ref{fig:ex2}(a) shows the
  problem at size $n_1=n_2=20$.

  Similar to the previous example, the experiments are performed for the problem size $n_1=n_2=100$
  at the inverse temperature $\beta=0.5$.  We start from the all $-1$ configuration and carry out
  $10000$ iterations for both SW and SWDF. The $\eta$ parameter of SWDF is $\eta=1/100$.  Figure
  \ref{fig:ex2}(b) shows that SW fails to introduce transitions between the $-1$ dominant and the
  $+1$ dominant profiles, while Figure \ref{fig:ex2}(c) demonstrates that SWDF explores both
  profiles with $100$ transitions in between.
\end{example}

\begin{figure}[h!]
  \centering
  \begin{tabular}{ccc}
    \includegraphics[scale=0.3]{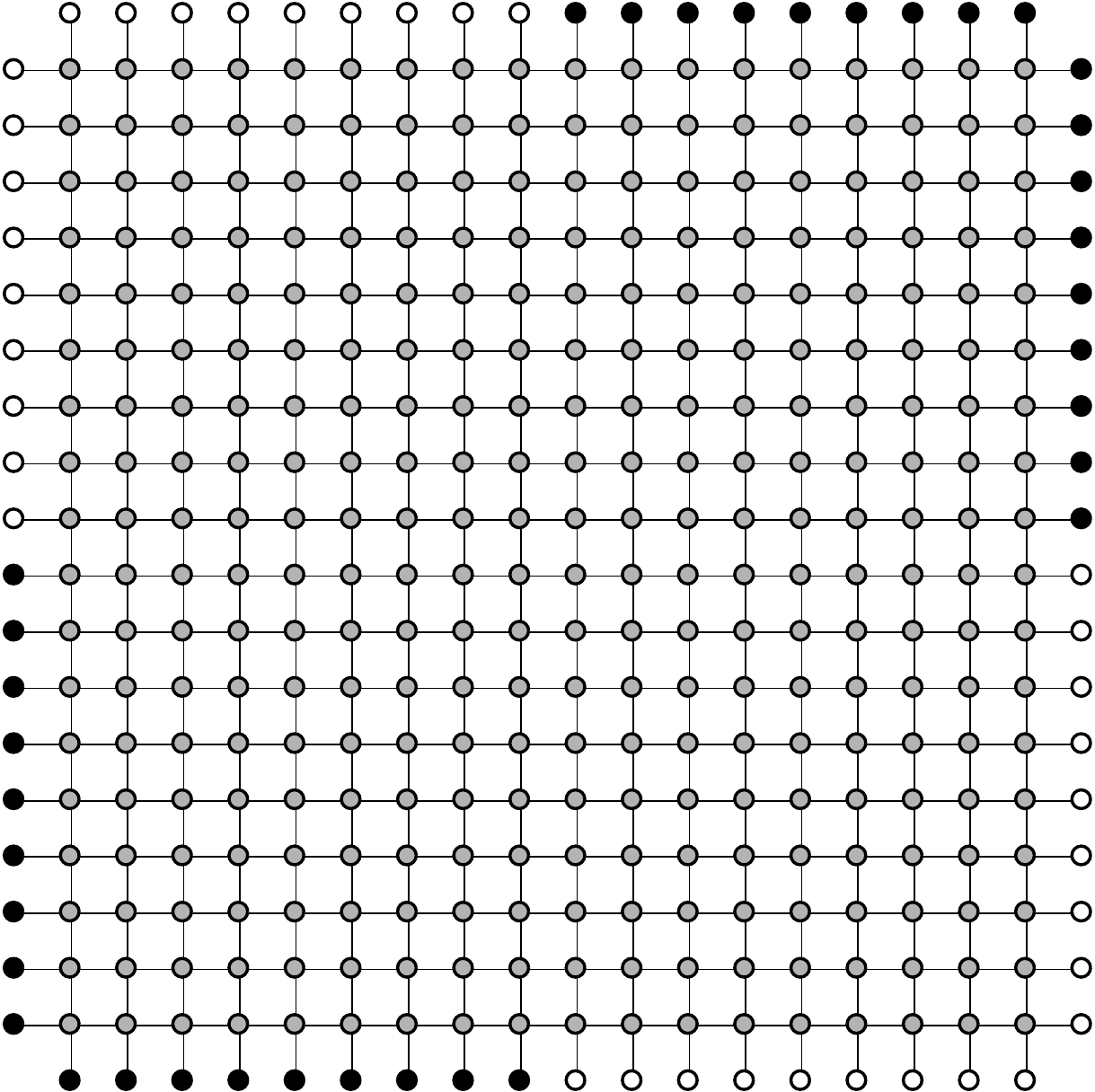} &
    \includegraphics[scale=0.3]{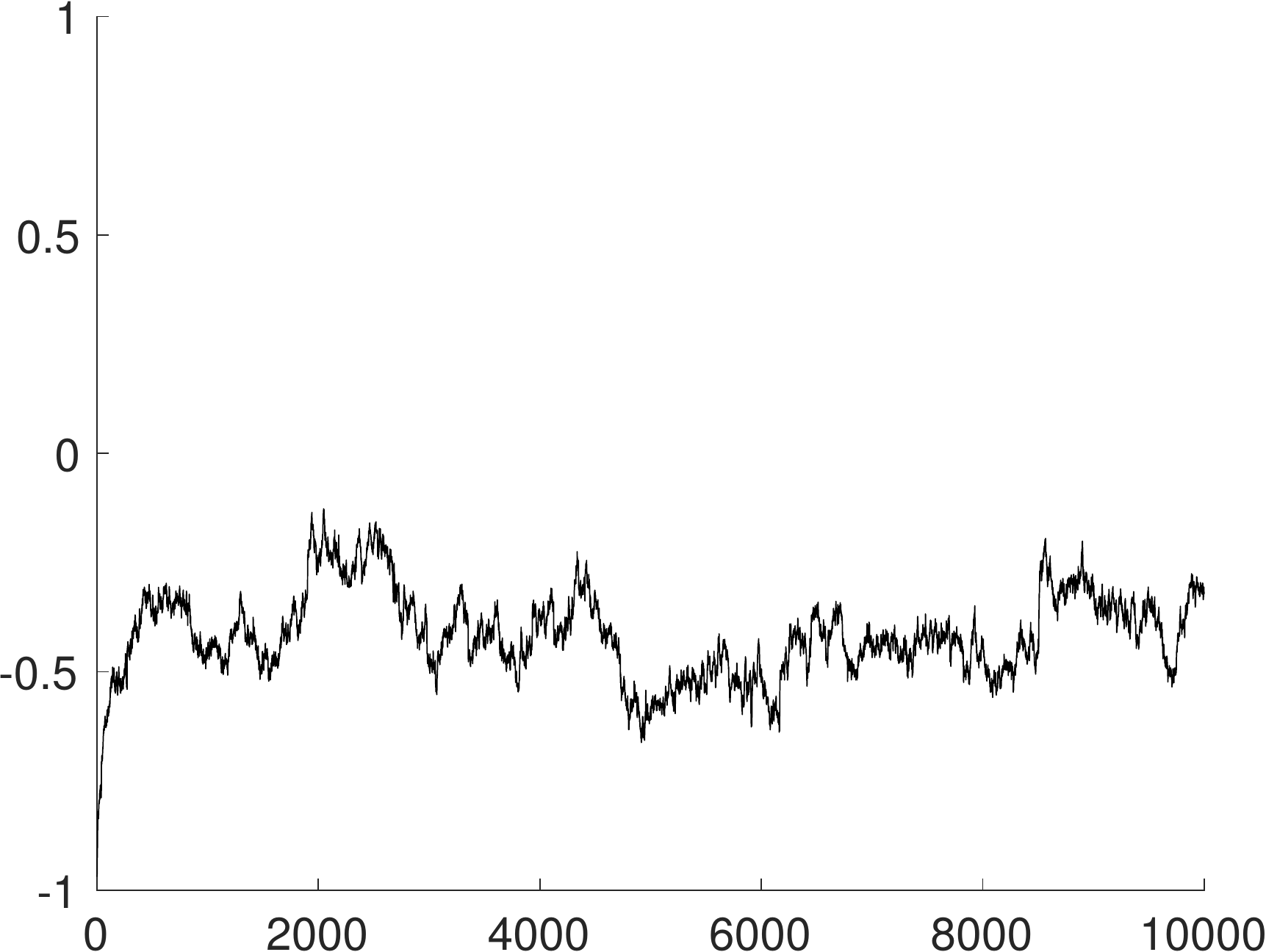} &
    \includegraphics[scale=0.3]{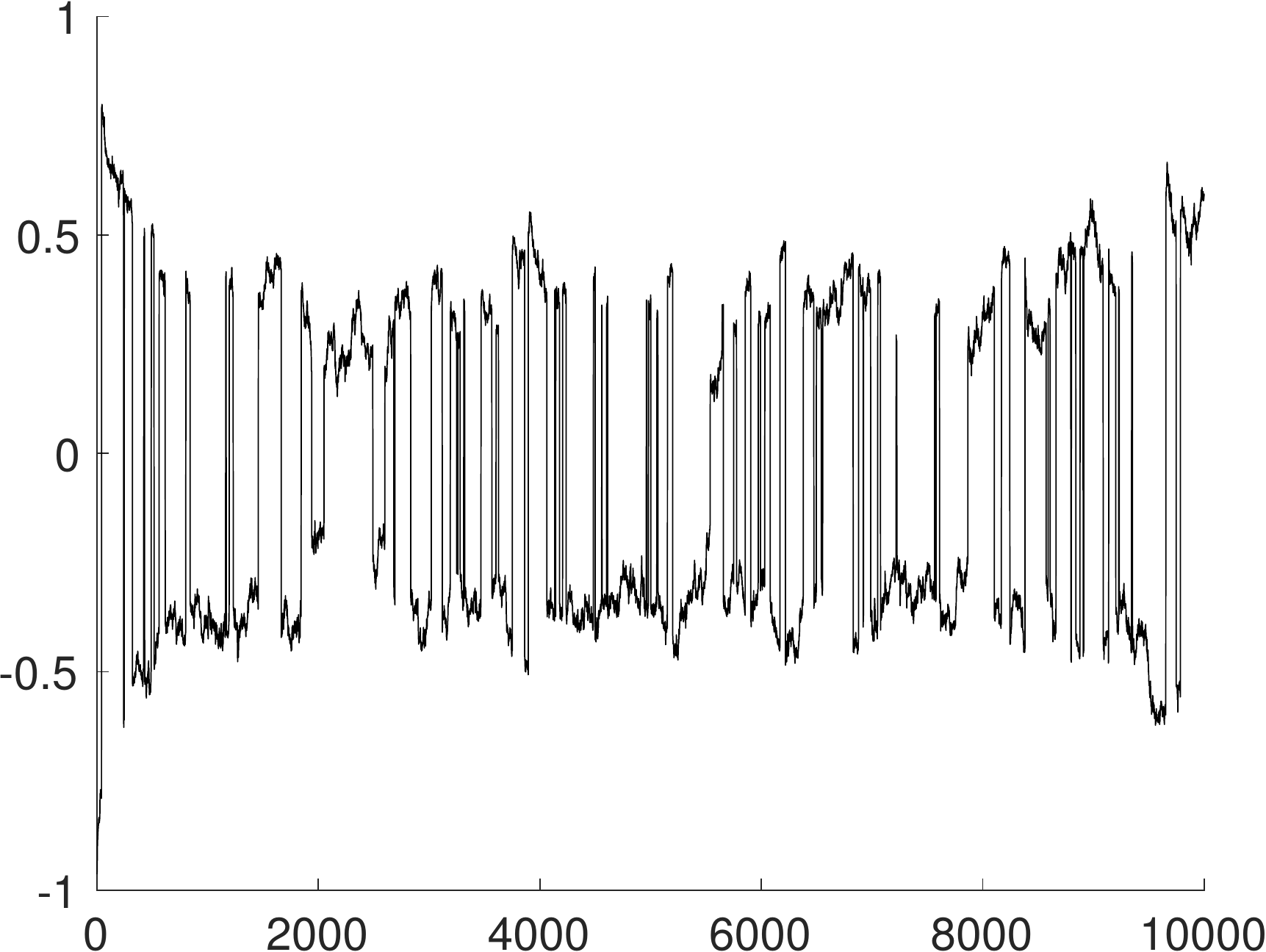} \\
    (a) & (b) & (c)
  \end{tabular}
  \caption{(a) The lattice along with the mixed boundary condition (black for $+1$ and white for
    $-1$). (b) The average spin value of the Swendsen-Wang algorithm. (c) The average spin value of
    the Swendsen-Wang algorithm with double flip.}
  \label{fig:ex2}
\end{figure}

%----------------------------------------------------------
\section{Double flip for approximately symmetric models}\label{sec:app}

The algorithm in Section \ref{sec:sym} is efficient but depends on exact symmetries. However, many
Ising models without exact symmetries also exhibit multiple dominant profiles such as in Figure
\ref{fig:front}. This section extends the double flip move to models with approximate symmetry.

%% \begin{figure}[h!]
%%   \centering
%%   \begin{tabular}{cc}
%%     \includegraphics[scale=0.2]{tmp.pdf} & \includegraphics[scale=0.2]{tmp.pdf}\\
%%     (a) & (b)
%%   \end{tabular}
%%   \caption{TODO}
%%   \label{fig:app}
%% \end{figure}

To take a more geometric viewpoint, assume that the Ising model is embedded in a domain
$\Omega\subset \R^2$ with the boundary denoted by $\p\Omega$.
\begin{itemize}
\item For each $i\in I$, $x_i$ is in the interior of $\Omega$. For each $b\in B$, $x_b$ is in
  $\p\Omega$.
\item The edges $ij$ or $ib$ in the edge set $E$ are segments between geometrically nearby vertices.
\item $\p\Omega = \p\Omega_+ \cup \p\Omega_-$. If $x_b \in \p\Omega_+$, then $f_b = 1$.  If
  $x_b\in\p\Omega_-$, then $f_b = -1$.
\item Assume that there is a {\em continuous} involution $\mu:\bar\Omega \rightarrow \bar\Omega$
  such that $\mu^2=\id$ and $\mu(\p\Omega_+) = \p\Omega_-$.
\end{itemize}

As $\mu$ is only defined as an involution of $\Omega$, in general $\{\mu(x_i)\}_{i\in I}
\not=\{x_i\}_{j\in I}$ and therefore $\mu$ does not directly introduce an involution on the set $I$
of interior vertices. To fix this, we introduce a {\em discrete} involution $m:I\rightarrow I$ such
that
\begin{equation}
  \label{eq:mu}
  x_{m(i)} \approx \mu(x_i).
\end{equation}
We shall discuss below how to construct $m$ based on $\mu$. For now, assuming the existence of $m$,
one can define a Metropolized double flip move for approximately symmetric models.
\begin{enumerate}
\item Define a spin configuration $t$ via $t_i = -s_{m(i)}$.
\item Evaluate $c = \min(1, \pV(t)/\pV(s))$.
\item Sample $u \in [0,1]$ uniformly. If $u\le c$, set $t$ to be the new spin
  configuration. Otherwise, keep $s$ as the spin configuration.
\end{enumerate}
Since $m$ is an involution and this is a Metropolized move, the following statement holds.
\begin{theorem}
  The Metropolized double flip move satisfies the detailed balance.
\end{theorem}
It can also be combined with the Swendsen-Wang move in the same way as described in Section
\ref{sec:sym}.
\begin{theorem}
  The Swendsen-Wang algorithm with Metropolized double flip satisfies the detailed balance.
\end{theorem}

The efficiency of this algorithm depends on the criteria of $\pV(t)/\pV(s)$ neither too small or too
large. This is in fact promoted by the condition \eqref{eq:mu}, since $ij\in E$ implies
\[
x_{m(i)} \approx \mu(x_i) \approx \mu(x_j) \approx x_{m(j)},
\]
where the second step uses the continuity of the domain involution $\mu$. Therefore, when $ij\in E$,
$m(i)m(j)$ is also likely to be an edge of $E$ with
\[
s_i s_j = (-1)^2 s_i s_j = t_{m(i)} t_{m(j)}.
\]
If this holds for most pairs $ij\in E$ and $ib\in E$, we have $\pV(t)/\pV(s)$ under control.

The remaining question is how to construct $m:I\rightarrow I$ so that \eqref{eq:mu} holds and $m^2 =
\id$. One possibility is to formulate this as a matching problem between the geometrically flipped
vertices $\{\mu(x_i)\}_{i\in I}$ and the original vertices $\{x_j\}_{j\in I}$, with a cost defined
using the $\ell_2$ or $\ell_\infty$ distance. Equivalently, this is an optimal transport problem
between the two distributions
\[
\sum_{i\in I} \delta_{\mu(x_i)}(\cdot)\quad\text{and}\quad \sum_{j\in I} \delta_{x_j}(\cdot).
\]
Once the matching (or the transport map) is available, we define $m(i)=j$ if $\mu(x_i)$
is matched with $x_j$. However, this approach has two technical difficulties.
\begin{itemize}
\item The computation cost of the matching or optimal transport algorithm
  \cite{kuhn1955hungarian,peyre2019computational} can be relatively high.
\item The involution condition $m^2 = \id$ is not guaranteed.
\end{itemize}
In the implementation, we adopt the following heuristic procedure. Assume without loss of generality
that the domain $\Omega$ is centered at the origin.
\begin{enumerate}
\item Order the interior vertices $\{x_j\}_{j\in I}$ based on their distances to the origin in the
  decreasing order. The distance is typically chosen to be the $\ell_\infty$ norm or the $\ell_2$
  norm.
\item Mark all vertices $j\in I$ as unpaired.
\item Scan the interior vertices in this ordered list. For each $x_j$, if $j$ is already paired, then
  skip. If not, find the unpaired $i$ such that $\mu(x_i)$ is closet to $x_j$, pair $i$ and $j$
  \[
  m(i):=j, \quad m(j):=i,
  \]
  and mark both $i$ and $j$ as paired. 
\end{enumerate}
The heuristic is that, by following the order of decreasing distance to the origin, the unpaired
vertices are forced to cluster near the center of the domain, thus reducing the overall transport
cost.

%% \begin{figure}[h!]
%%   \centering
%%   \begin{tabular}{cc}
%%     \includegraphics[scale=0.2]{tmp.pdf} & \includegraphics[scale=0.2]{tmp.pdf}\\
%%     (a) & (b)
%%   \end{tabular}
%%   \caption{TODO}
%%   \label{fig:appsnap}
%% \end{figure}

%=================
\begin{figure}[h!]
  \centering
  \begin{tabular}{cc}
    \includegraphics[scale=0.3]{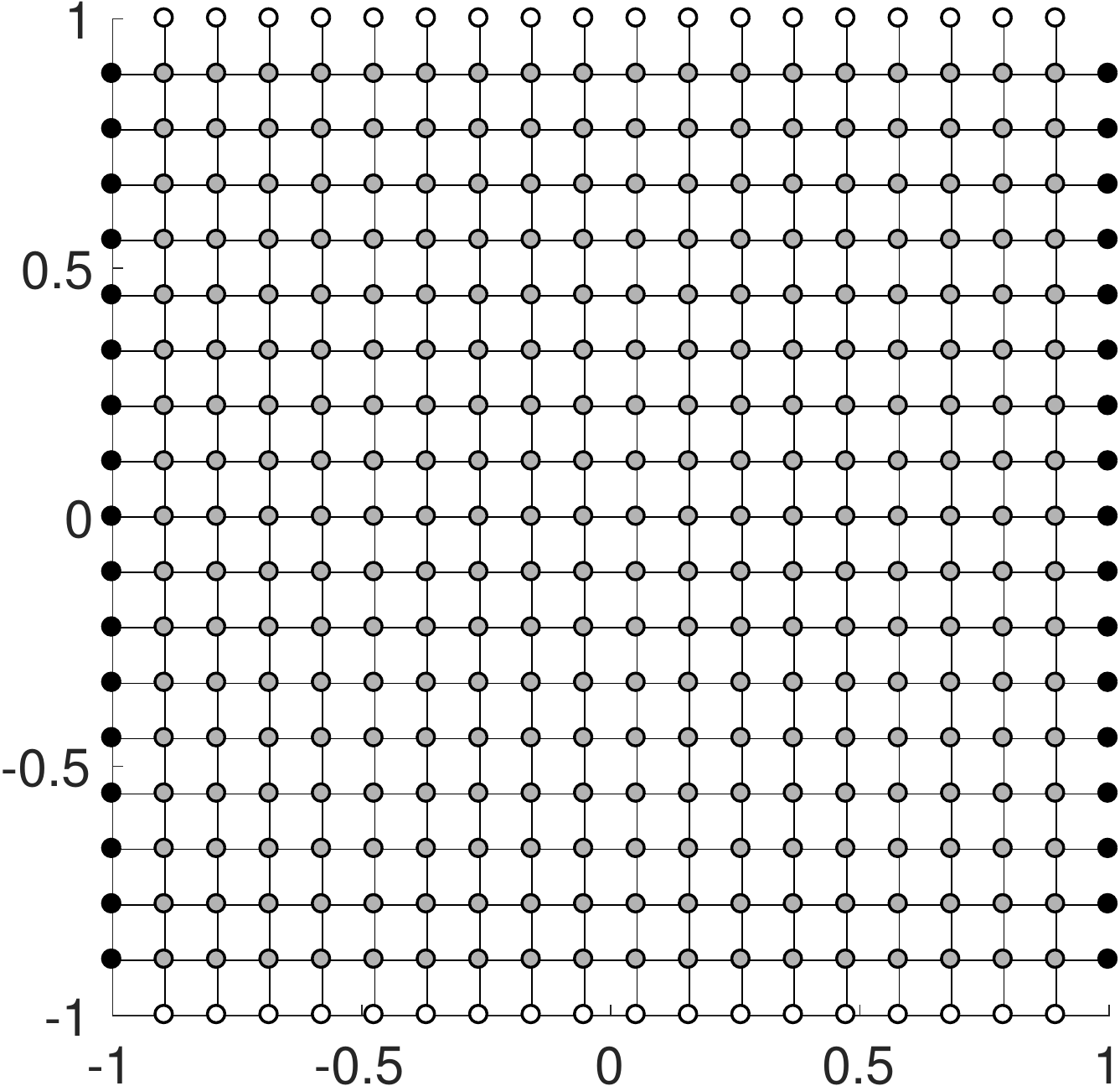} &    \includegraphics[scale=0.3]{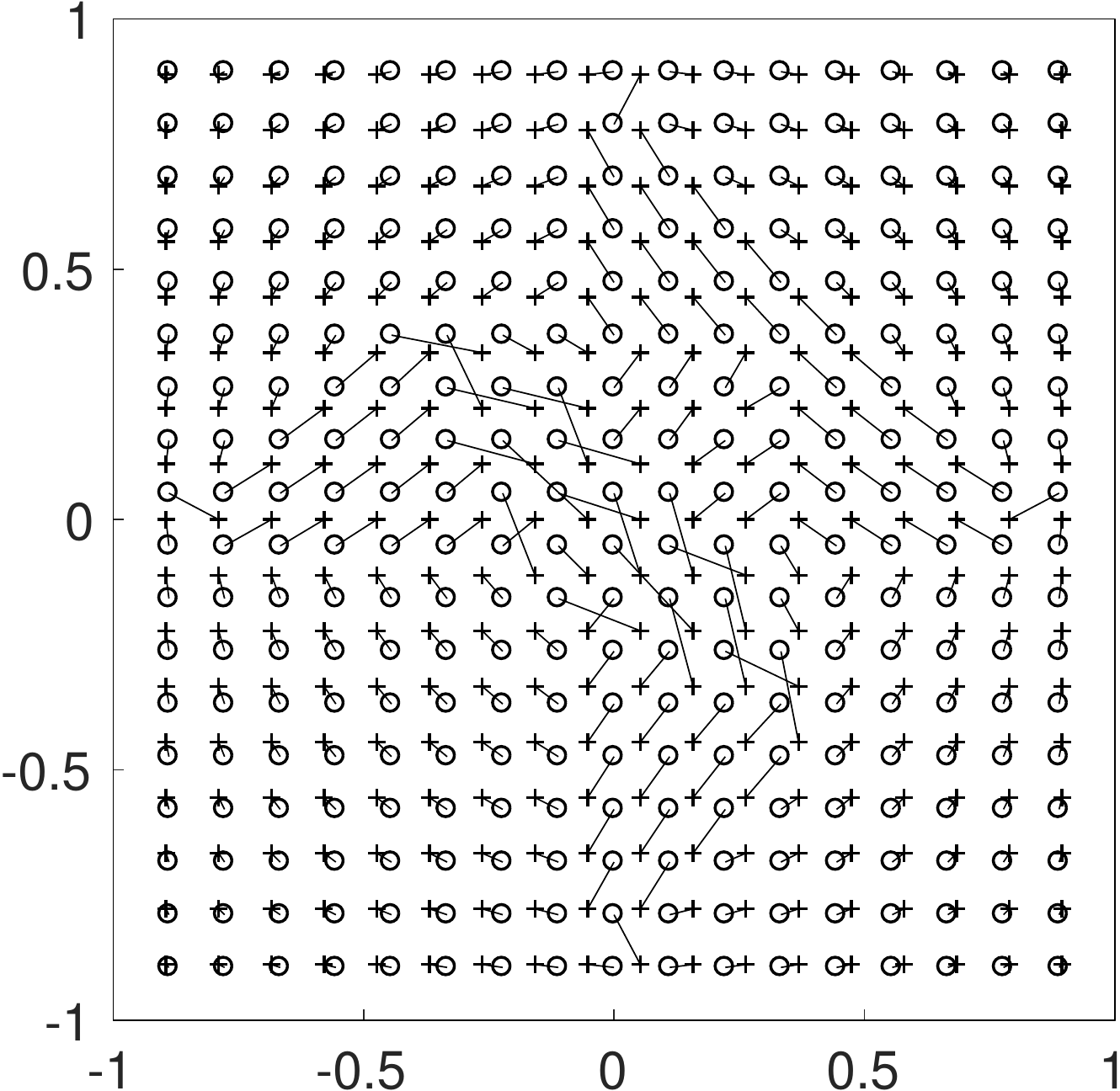}\\
    (a) & (b) \\
    \includegraphics[scale=0.3]{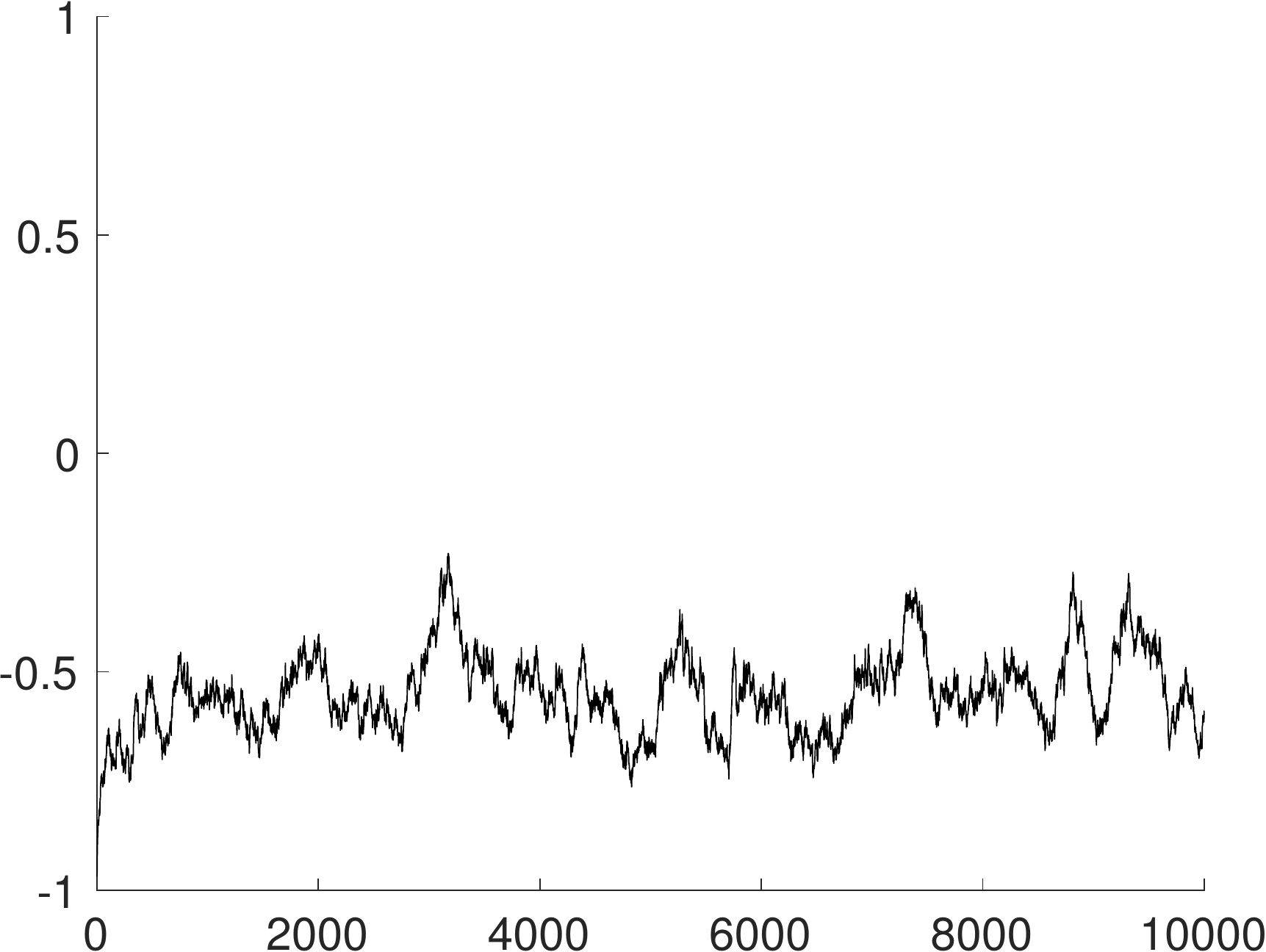} &    \includegraphics[scale=0.3]{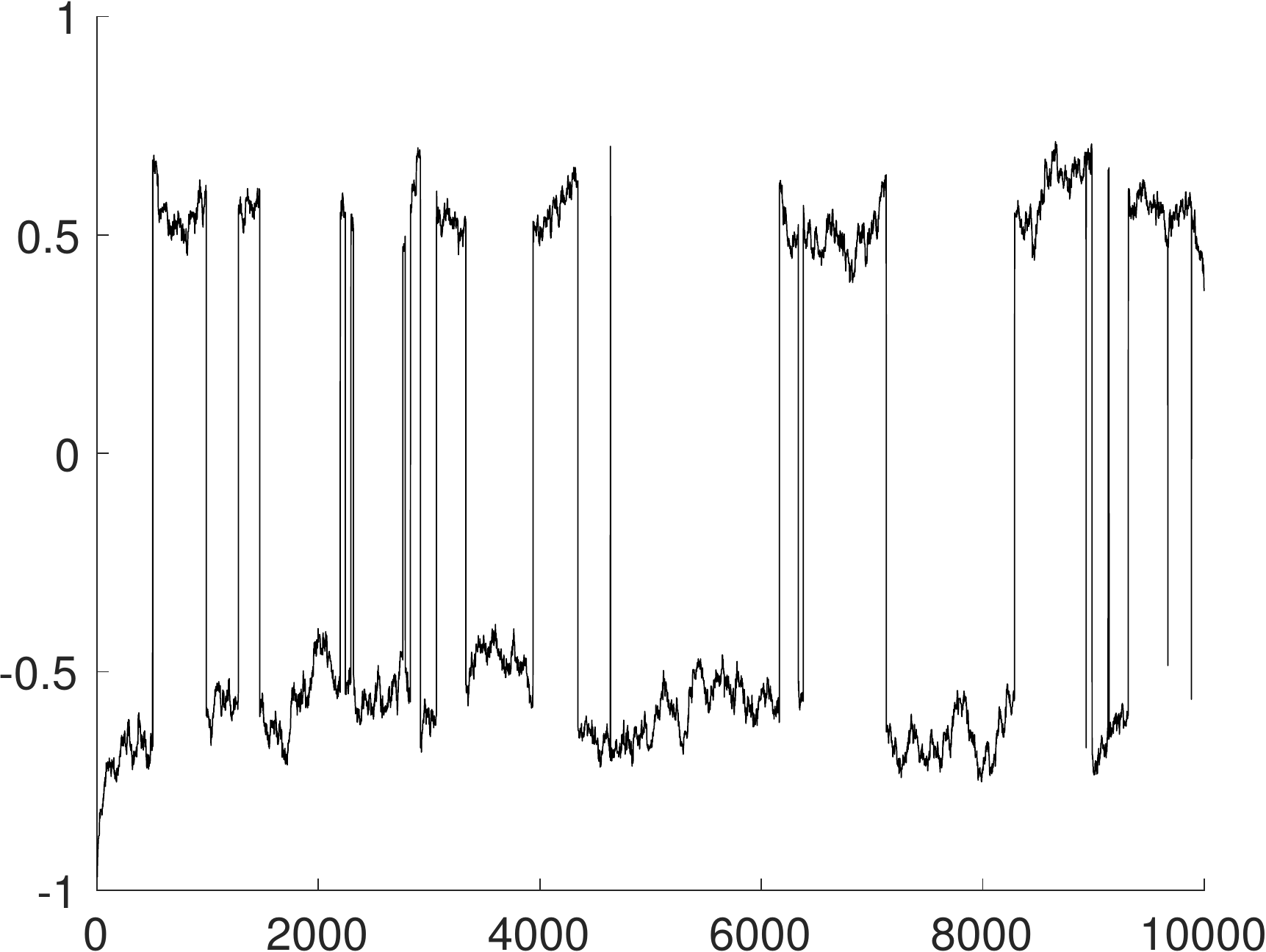} \\
    (c) & (d)
  \end{tabular}
  \caption{(a) The rectangular lattice along with the mixed boundary condition (black for $+1$ and
    white for $-1$).  (b) The transport map from $\{\mu(x_i)\}_{i\in I}$ (marked with $+$) to
    $\{x_j\}_{j\in I}$ (marked with $\circ$). (c) The average spin value of the Swendsen-Wang
    algorithm. (d) The average spin value of the Swendsen-Wang algorithm with double flip.}
  \label{fig:ex3}
\end{figure}

Below we compare the performance of the Swendsen-Wang algorithm (SW) and Swendsen-Wang with
Metropolized double flip (SWDF) using three examples.

\begin{example}
The Ising model is a rectangular lattice where the number of rows and columns are different, as
shown in Figure \ref{fig:ex3}. The mixed boundary condition is $+1$ at the two vertical sides and
$-1$ at the two horizontal sides. The diagonal reflection is no longer an exact symmetry.  Figure
\ref{fig:ex3}(a) shows the system with size $20\times 19$.  Figure \ref{fig:ex3}(b) plots the
transport map between $\{\mu(x_i)\}_{i\in I}$ (marked with $+$) to $\{x_j\}_{j\in I}$ (marked with
$\circ$). As shown, the transport map is quite local, demonstrating the efficiency of the heuristic
matching procedure.

The experiments are performed for the problem size $100\times 99$ at the inverse temperature
$\beta=0.5$. We start from the all $-1$ configuration and carry out $10000$ iterations for both SW
and SWDF. The $\eta$ parameter of SWDF is $\eta=1/3$. Figure \ref{fig:ex3}(c) shows that SW fails to
introduce transitions between the $-1$ dominant and the $+1$ dominant profiles, while Figure
\ref{fig:ex3}(d) demonstrates that SWDF explores both profiles with $41$ transitions out of about
$3000$ trials.
\end{example}

%=================
\begin{figure}[h!]
  \centering
  \begin{tabular}{cc}
    \includegraphics[scale=0.3]{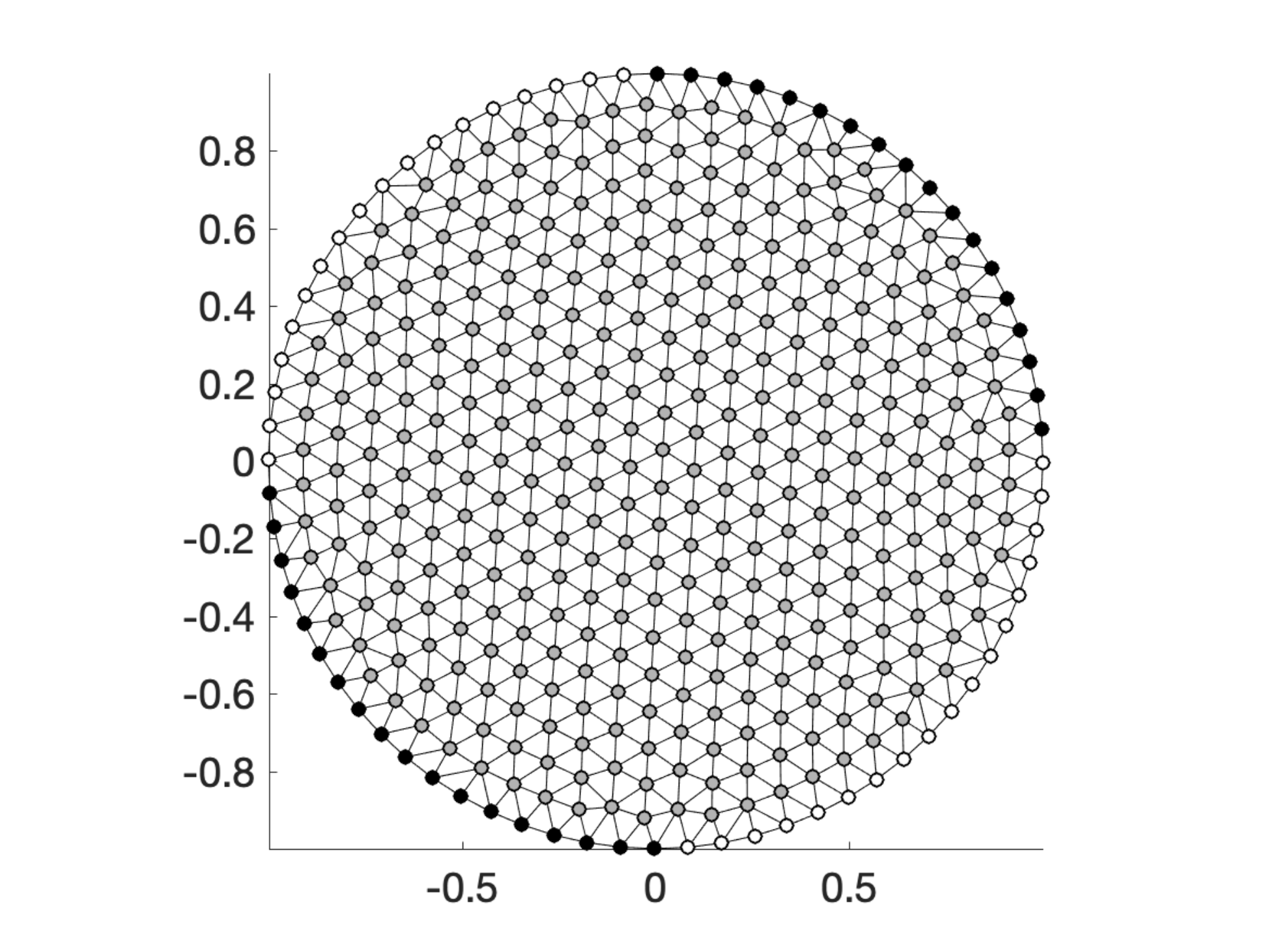} &    \includegraphics[scale=0.3]{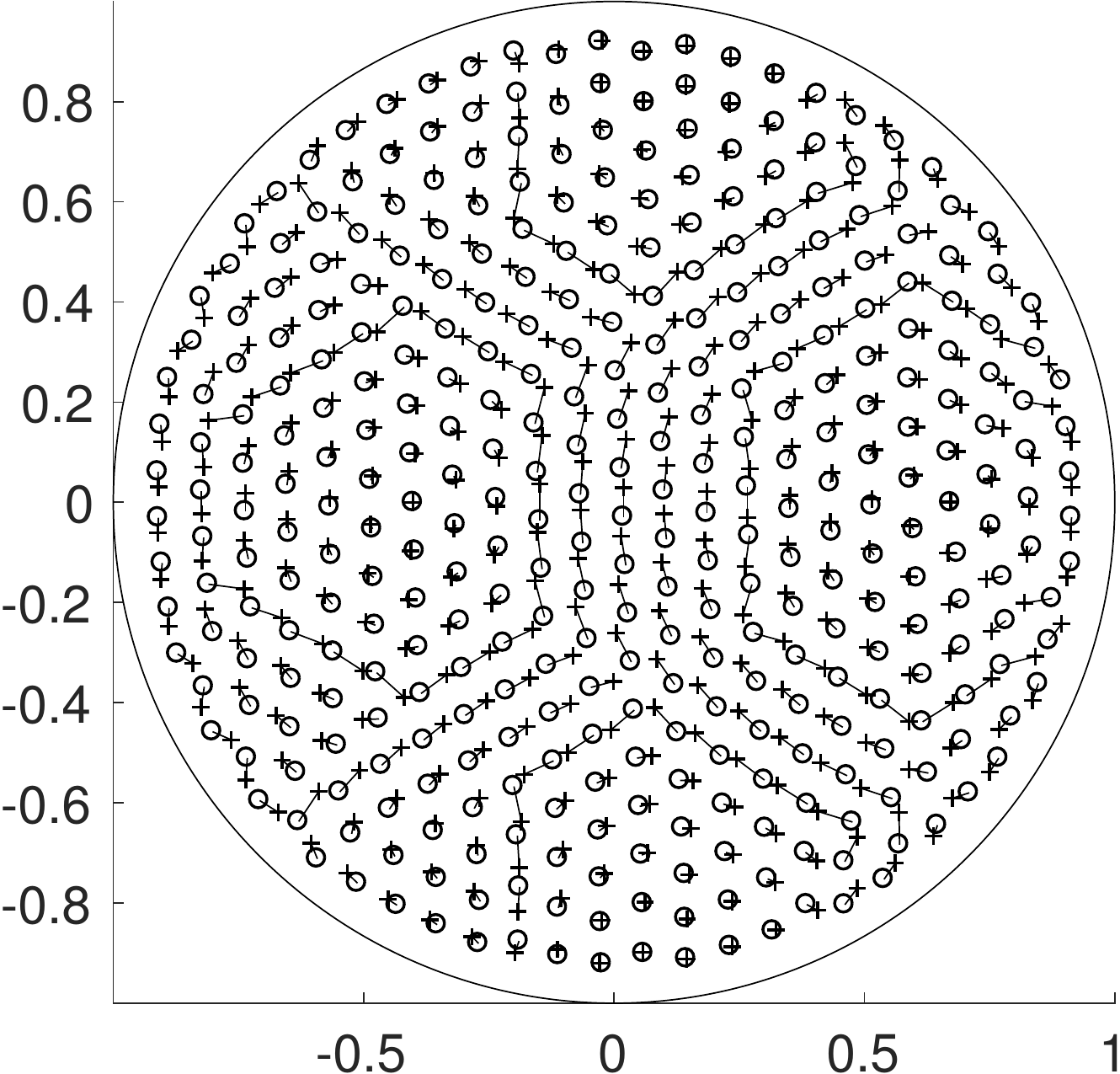}\\
    (a) & (b) \\
    \includegraphics[scale=0.3]{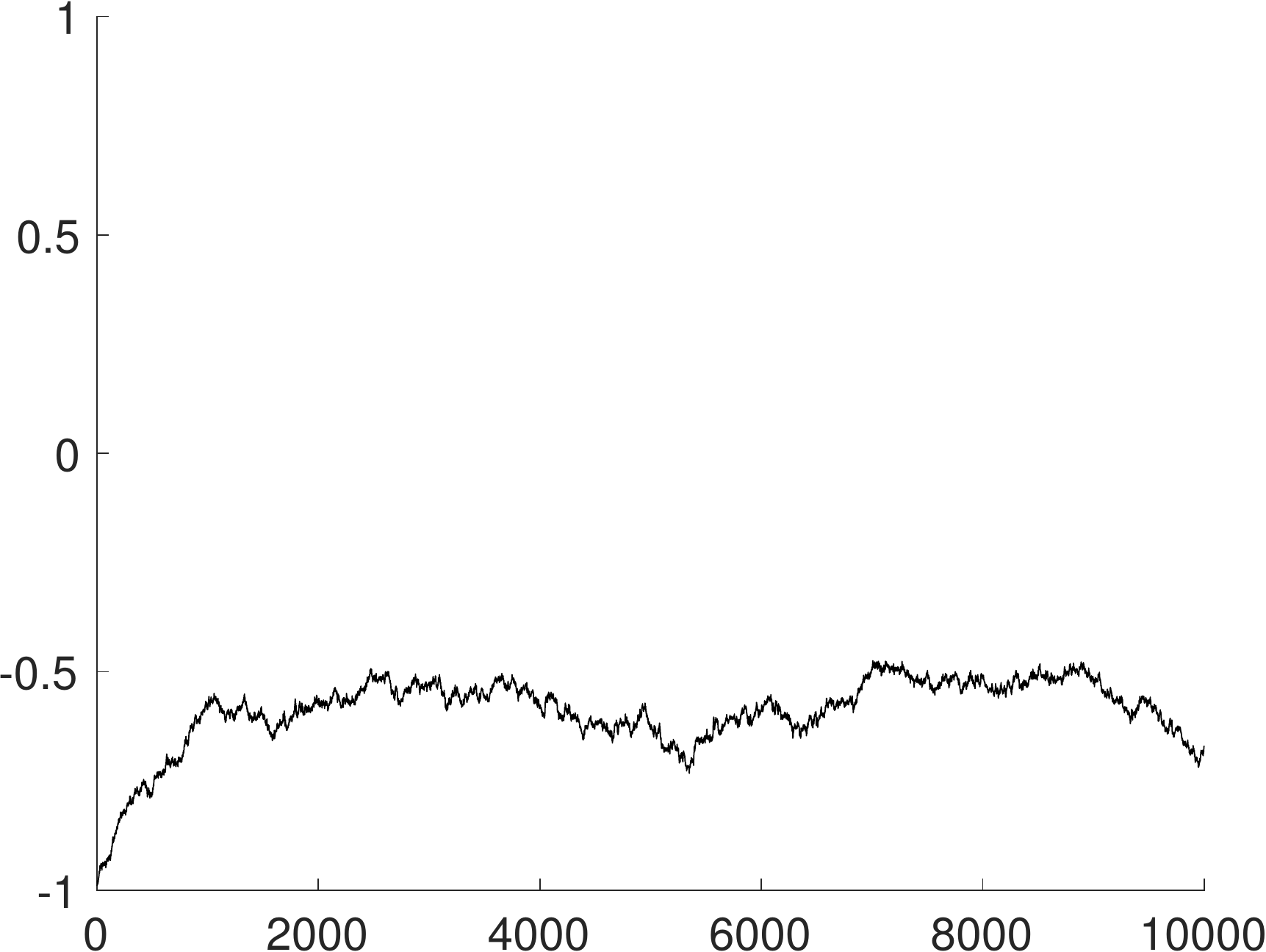} &    \includegraphics[scale=0.3]{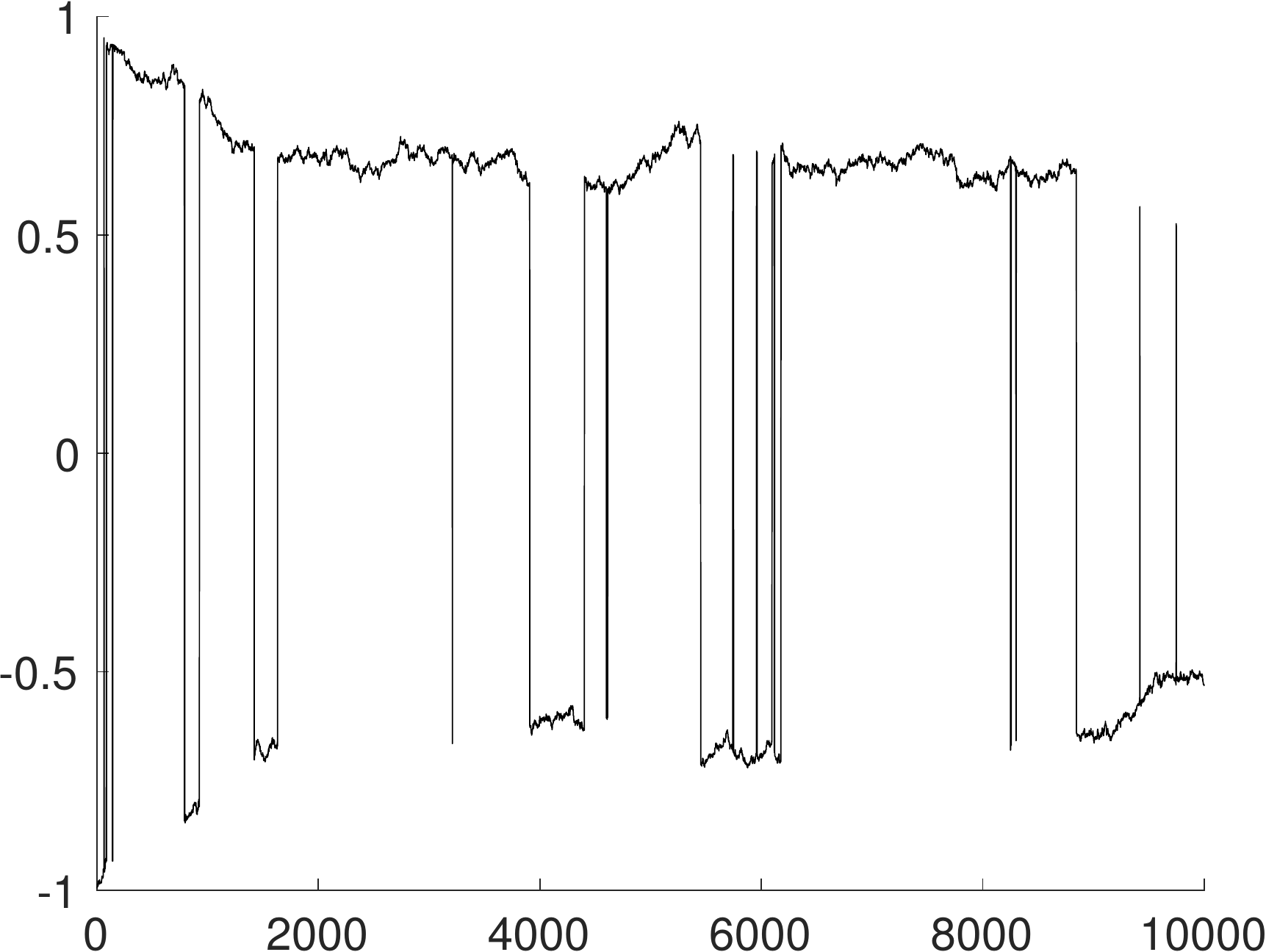} \\
    (c) & (d)
  \end{tabular}
  \caption{(a) The lattice along with the mixed boundary condition (black for $+1$ and white for
    $-1$).  (b) The transport map from $\{\mu(x_i)\}_{i\in I}$ (marked with $+$) to $\{x_j\}_{j\in
      I}$ (marked with $\circ$). (c) The average spin value of the Swendsen-Wang algorithm. (d) The
    average spin value of the Swendsen-Wang algorithm with double flip.}
  \label{fig:ex4}
\end{figure}

\begin{example}
The Ising model is a random quasi-uniform triangular lattice supported on the unit disk, as shown in
Figure \ref{fig:ex4}. The mixed boundary condition is equal to $+1$ in the first and third quadrants
but $-1$ in the second and fourth quadrants. The problem does not have strict rotation and
reflection symmetry due to the random triangulation. Figure \ref{fig:ex4}(a) shows the triangulation
with mesh size $h=0.1$. Figure \ref{fig:ex4}(b) gives the transport map between $\{\mu(x_i)\}_{i\in I}$
(marked with $+$) to $\{x_j\}_{j\in I}$ (marked with $\circ$), which is quite local.

The experiments are performed with a finer triangulation with mesh size $h=0.05$ at the inverse
temperature $\beta=0.5$. We start from the all $-1$ configuration and carry out $10000$ iterations
for both SW and SWDF. The $\eta$ parameter of SWDF is $\eta=1/3$. Figure \ref{fig:ex4}(c) shows that
SW fails to introduce transitions between the $-1$ dominant and the $+1$ dominant profiles, while
Figure \ref{fig:ex4}(d) demonstrates that SWDF explores both profiles with $48$ transitions
out of about $3000$ trials.
\end{example}

%=================
\begin{figure}[h!]
  \centering
  \begin{tabular}{cc}
    \includegraphics[scale=0.3]{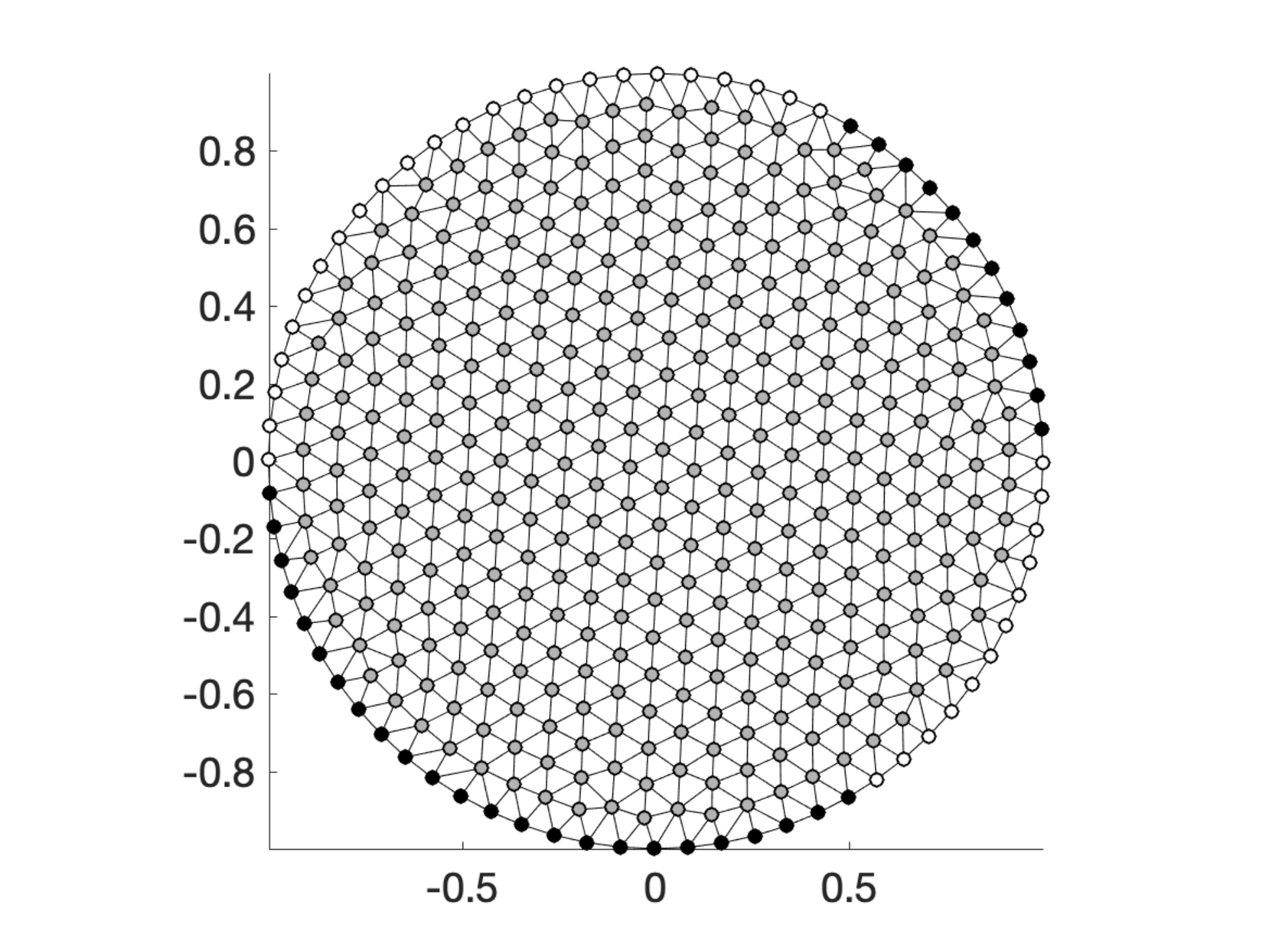} &    \includegraphics[scale=0.3]{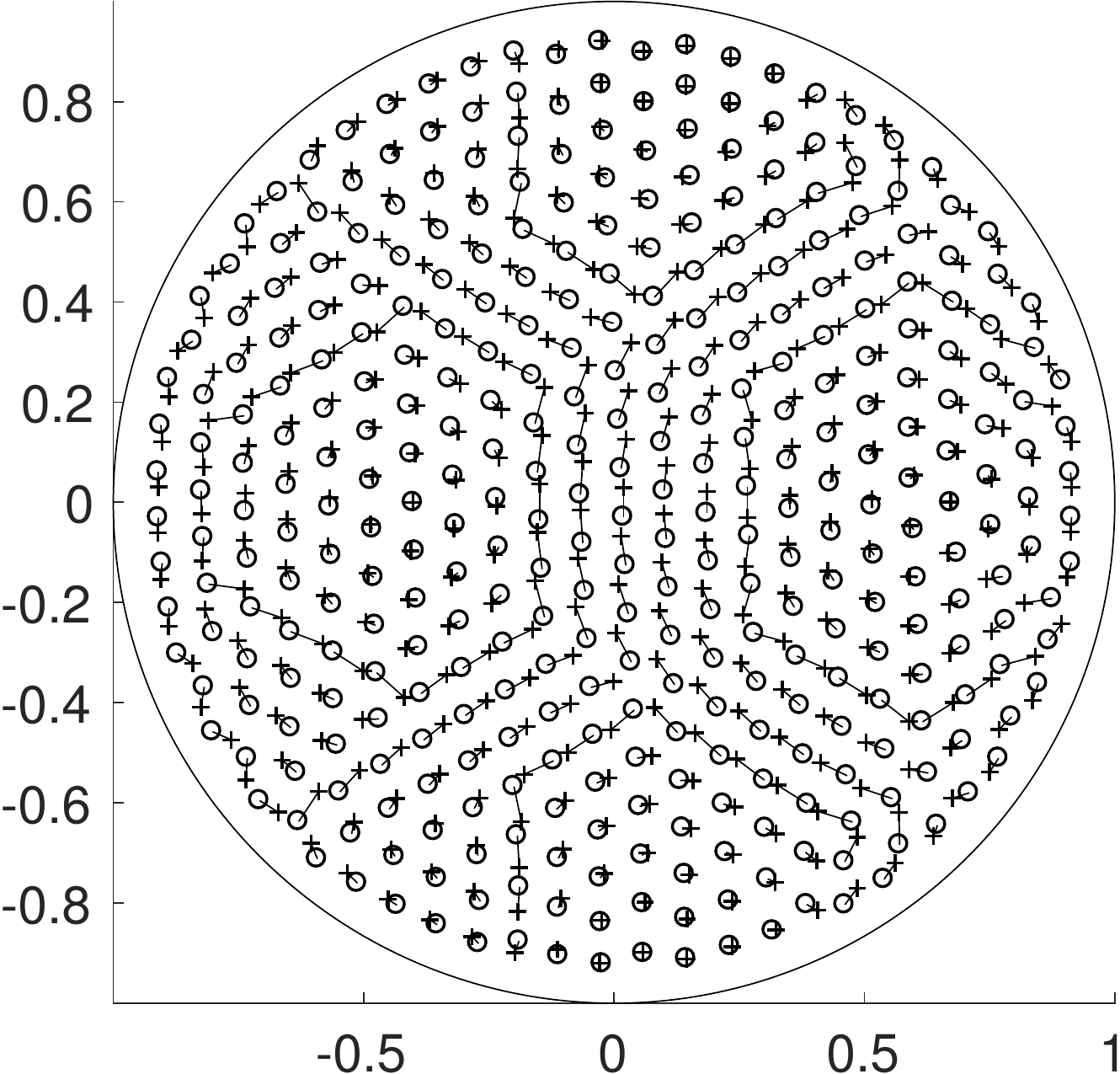}\\
    (a) & (b) \\
    \includegraphics[scale=0.3]{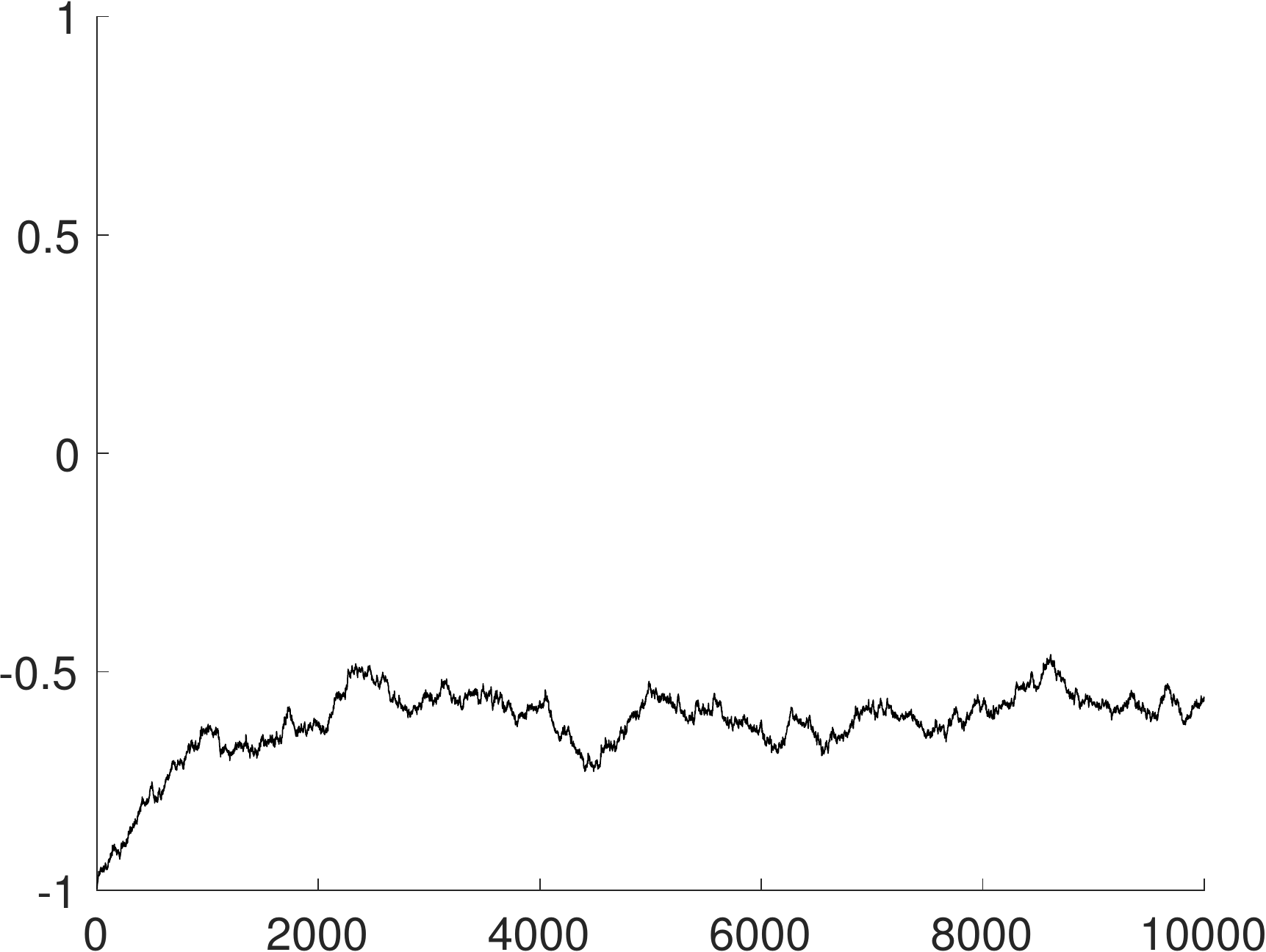} &    \includegraphics[scale=0.3]{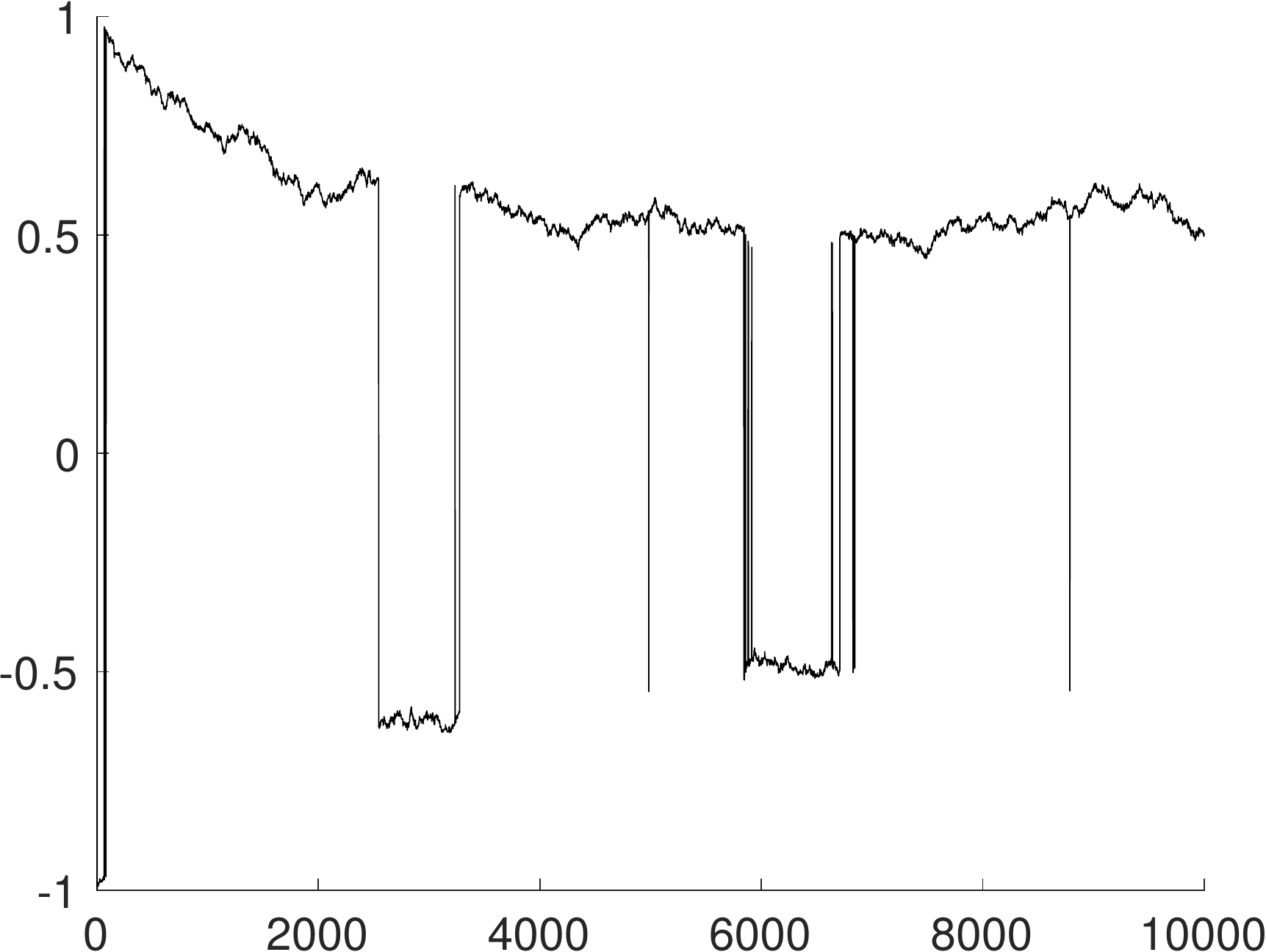} \\
    (c) & (d)
  \end{tabular}
  \caption{(a) The lattice along with the mixed boundary condition (black for $+1$ and white for
    $-1$).  (b) The transport map from $\{\mu(x_i)\}_{i\in I}$ (marked with $+$) to $\{x_j\}_{j\in
      I}$ (marked with $\circ$). (c) The average spin value of the Swendsen-Wang algorithm. (d) The
    average spin value of the Swendsen-Wang algorithm with double flip.}
  \label{fig:ex5}
\end{figure}

\begin{example}
The Ising model is again a random quasi-uniform triangular lattice supported on the unit disk. The
mixed boundary condition is equal to $+1$ on the two arcs with angle in $[0,\pi/3]$ and
$[\pi,5\pi/3]$ but $-1$ on the remaining two arcs. Due to the random triangulation, the problem does
not have strict rotation and reflection symmetry. Figure \ref{fig:ex5}(a) shows the triangulation
with mesh size $h=0.1$. Figure \ref{fig:ex5}(b) plots the transport map between $\{\mu(x_i)\}_{i\in
  I}$ (marked with $+$) to $\{x_j\}_{j\in I}$ (marked with $\circ$).

The experiments are performed with a finer triangulation with mesh size $h=0.05$ at the inverse
temperature $\beta=0.5$. We start from the all $-1$ configuration and carry out $10000$ iterations
for both SW and SWDF. The $\eta$ parameter of SWDF is $\eta=1/3$. Figure \ref{fig:ex5}(c) shows that
SW fails to introduce transitions between the $-1$ dominant and the $+1$ dominant profiles, while
Figure \ref{fig:ex5}(d) demonstrates that SWDF explores both profiles with $35$ transitions
out of about $3000$ trials.
\end{example}

%----------------------------------------------------------
\section{Discussions}\label{sec:disc}

This note introduces the double flip move for accelerating the Swendsen-Wang algorithm for Ising
models with mixed boundary conditions. We consider both symmetric and approximately symmetric
models. In both cases, we prove the detailed balance and demonstrated its efficiency in introducing
explicit transitions between different dominant profiles.

There are many unanswered questions. Regarding the symmetric models, one question is to prove a
polynomial mixing time for the examples in Section \ref{sec:sym}. Regarding the approximately
symmetric models, there are more open questions.
\begin{itemize}
\item Is there a fast matching or optimal transport algorithm that ensures $m^2 = \id$?
\item Better heuristic procedures for constructing the matching between $\{\mu(x_i)\}_{i\in I}$ and
  $\{x_j\}_{j\in I}$?
\item Can we bound the acceptance ratio of the Metropolized double flip move under certain
  assumptions of the approximate symmetry?
\item Proving a rapid mixing result for any approximately symmetric model in Section \ref{sec:app}.
\item The approximate matching is carried out for the interior vertices in this note. However, it
  can be carried out for the edges alternatively.
\end{itemize}

%----------------------------------------------------------
\section*{Acknowledgements}
The author thanks Sourav Chatterjee for discussions and for introducing
\cite{chatterjee2020speeding}.

\bibliographystyle{abbrv}

\bibliography{ref}

\end{document}